\documentclass[11pt]{amsart}
\usepackage{amsfonts}
\usepackage{amsfonts}
\usepackage{amsfonts}
\usepackage{amsfonts}
\theoremstyle{plain}
\newtheorem{Thm}{Theorem}

\newtheorem{Cor}[Thm]{Corollary}
\newtheorem{Prop}[Thm]{Proposition}

\newtheorem{Rk}[Thm]{Remark}

\begin{document}
\large

\title[sharp thresholds for schr\"{o}dinger system]
{Sharp thresholds of blow-up and global existence for the coupled
nonlinear Schr\"{o}dinger system}
\author{ Li Ma and Lin Zhao}
\address{Li Ma, Department of Mathematical Sciences, Tsinghua University,
 Peking 100084, P. R. China}

\email{lma@math.tsinghua.edu.cn}

\thanks{The research is partially supported by the National Natural Science
Foundation of China 10631020 and SRFDP 20060003002}

\begin{abstract}

In this paper, we establish two new types of invariant sets for the
coupled nonlinear Schr\"{o}dinger system on $\mathbb{R}^n$, and
derive two sharp thresholds of blow-up and global existence for its
solution. Some analogous results for the nonlinear Schr\"{o}dinger
system posed on the hyperbolic space $\mathbb{H}^n$ and on the
standard 2-sphere $\mathbb{S}^2$ are also presented. Our arguments
and constructions are improvements of some previous works on this
direction. At the end, we give some heuristic analysis about the
strong instability of the solitary waves.\\



\end{abstract}

\maketitle
\date{4-16-2007}

\section{Introduction}

In this paper, we establish two new types of invariant sets for the
$N$-coupled nonlinear Schr\"{o}dinger system on $\mathbb{R}^n$ given
by
\begin{align}
\label{system}\left\{\begin{array}{ll}
-i\partial_t\phi_j=\Delta\psi_j+\mu_j|\phi_j|^{p-1}\phi_j+\sum_{i\neq
j}\beta_{ij}|\phi_i|^{(p+1)/2}|\phi_j|^{(p-3)/2}\phi_j,\\
\phi_j=\phi_j(t,x)\in \mathbb{C},\ \ x\in \mathbb{R}^n,\ \ t>0,\ \ j=1,...,N,\\
\phi_j(0,x)=\phi_{0j}(x),\ \ \phi_{0j}: \mathbb{R}^n\rightarrow
\mathbb{C},
\end{array}
\right.
\end{align}
where $1\leq p<1+4/(n-2)^+$ (we use the convention:
$4/(n-2)^+=\infty$ when $n=1,2$, and $(n-2)^+=n-2$ when $n\geq 3$),
$\mu_j>0$'s are positive constants and $\beta_{ij}$'s are coupling
constants subjected to $\beta_{ij}=\beta_{ji}$. Based on our new
invariant sets, we then derive two sharp thresholds of blow-up and
global existence for the solutions. We point out that our results
have no restriction on the dimension $n$, which plays an important
role in the previous related studies \cite{H.Be}. We also give the
sharp thresholds when (\ref{system}) is considered on the hyperbolic
space $\mathbb{H}^n$ and on the standard 2-sphere $\mathbb{S}^2$.
These results rely heavily on the geometric structure of the
manifolds and behave very differently from the ones considered on
$\mathbb{R}^n$. At the end, we give some heuristic analysis about
the strong instability of the solitary waves.

The system (\ref{system}) has applications in many physical
problems, especially in nonlinear optics. Physically, the solution
$\phi_j$ denotes the $j^{th}$ component of the beam in Kerr-like
photo-refractive media (cf.\cite{N.A}). The positive constant
$\mu_j$ is for self-focusing in the $j^{th}$ component of the beam.
The coupling constant $\beta_{ij}$ is the interaction between the
$i^{th}$ and the $j^{th}$ component of the beam. We refer to
\cite{H.B} for more precision on the meaning of the constants. When
the spatial dimension $n\leq 3$, there are many analytical and
numerical results on the system. We shall quote the recent works
\cite{KW.C,FT.H,T.K,TC.L1,TC.L2,B.S}, where a comprehensive list of
references on this subject can be found. However, there are few
works describing the blow-up phenomena of the solution. Hereafter,
we focus on the blow-up analysis for the system (\ref{system}) when
$\beta_{ij}=\beta_{ji}$. For notational simplicity, we write
$\Phi_0=(\phi_{01},...,\phi_{0N})$ as the initial data and
$\Phi=(\phi_{1},...,\phi_{N})$ as the solution. We denote
$$\|\Phi\|_p:=(\sum_{j=1}^N\int_{\mathbb{R}^n} |\phi_j|^p)^{1/p}$$
for $1\leq p<\infty$ and
$$\|\nabla\Phi\|_2:=(\sum_{j=1}^N\int_{\mathbb{R}^n}
|\nabla\phi_j|^2)^{1/2}.$$ We define the testing functional
$$
\mathcal{P}(\Phi):=\sum_{j=1}^N\mu_j\int_{\mathbb{R}^n}|\phi_j|^{p+1}+
\sum_{i,j=1}^N\beta_{ij}\int_{\mathbb{R}^n}|\phi_i|^{(p+1)/2}|\phi_j|^{(p+1)/2}.
$$

The local existence theorem for the single Schr\"{o}dinger equation
in $H^1(\mathbb{R}^n)$ (see \cite{T.C,J.G}) still holds true for the
Schr\"{o}dinger system (\ref{system}). In fact, by solving the
equivalent integral system
\begin{align*}
\phi_j&=e^{it\Delta}\phi_{0j}+i\mu_j\int_0^te^{i(t-s)\Delta}|\phi_j|^{p-1}\phi_j(s)ds\\
&\quad+i\sum_{i\neq
j}\beta_{ij}\int_0^te^{i(t-s)\Delta}|\phi_i|^{(p+1)/2}|\phi_j|^{(p-3)/2}\phi_j(s)ds
\end{align*}
in the space
$$
(H^1(\mathbb{R}^n))^N=\underbrace{H^1(\mathbb{R}^n)\times...\times
H^1(\mathbb{R}^n)}_{N}
$$
with a standard Picard iteration method as in
\cite{T.C,J.G}, one gets easily the following proposition.

\begin{Prop} (Local Existence)\label{localprop}
Assume that $1\leq p<1+4/(n-2)^+$. Then for any $\Phi_0\in
(H^1(\mathbb{R}^n))^N$, there exists a $T>0$ and a unique solution
$\Phi\in \mathcal{C}([0,T),(H^1(\mathbb{R}^n))^N)$ such that either
$T=\infty$ or else $T<\infty$ and
$\|\nabla\Phi\|_2\rightarrow\infty$ as $t\rightarrow T$.
\end{Prop}

When $\beta_{ij}=\beta_{ji}$, the system (\ref{system}) admits the
mass and the energy conservation laws in the space
$(H^1(\mathbb{R}^n))^N$, which are stated in (\ref{M}) and (\ref{E})
below.

\emph{Mass ($L^2$ norm):}
\begin{align}
M(\Phi):=\|\Phi\|_2=M(\Phi_0);\label{M}
\end{align}

\emph{Energy:}
\begin{align}
E(\Phi):=\frac{1}{2}\|\nabla\Phi\|_2^2-\frac{1}{p+1}\mathcal{P}(\Phi)=E(\Phi_0).\label{E}
\end{align}

Furthermore, let $\rho>0$ be a $\mathcal{C}^4$ real function on
$\mathbb{R}^n$ (independent of $t$), and then for
$$J(t):=\sum_{j=1}^N\int_{\mathbb{R}^n} \rho(x)|\phi_j(t,x)|^2,$$
we have
\begin{align}
J'(t)=2\Im\sum_{j=1}^N\int_{\mathbb{R}^n}(\nabla\phi_j\cdot\nabla\rho)\bar{\phi_j}\label{J'}
\end{align}
and
\begin{align}
J^{''}(t)&=4\sum_{j=1}^N\int_{\mathbb{R}^n}
D^2\rho(\nabla\phi_j,\nabla\bar{\phi_j})-\sum_{j=1}^N\int_{\mathbb{R}^n}
(\Delta^2\rho)|\phi_j|^2\label{J"}\\
&\quad-2\frac{p-1}{p+1}\sum_{j=1}^N\mu_j\int_{\mathbb{R}^n}
(\Delta\rho)|\phi_j|^{p+1}\nonumber\\
&\quad-2\frac{p-1}{p+1}\sum_{i,j=1}^N\beta_{ij}\int_{\mathbb{R}^n}
(\Delta\rho)|\phi_i|^{(p+1)/2}|\phi_j|^{(p+1)/2}\nonumber
\end{align}
under the assumption $\beta_{ij}=\beta_{ji}$. Especially, if we
choose $\rho(x)=|x|^2$ (see \cite{Gl} and \cite{Og1}), we then get
that
\begin{align*}
J'(t)=4\Im\sum_{j=1}^N\int_{\mathbb{R}^n}(\nabla\phi_j\cdot
x)\bar{\phi_j}
\end{align*}
and
\begin{align*}
J^{''}(t)=16Q(\Phi),
\end{align*}
where
\begin{align}
Q(\Phi):=\frac{1}{2}\|\nabla\Phi\|_2^2-\frac{n(p-1)}{4(p+1)}\mathcal{P}(\Phi).\label{Q}
\end{align}

Applying the classical energy argument, one has for $p<1+4/n$, the
solution of (\ref{system}) exists globally. In fact, assuming
$|E(\Phi_0)|<\infty$ and thanks to the Gagliardo-Nirenberg
inequality on $\mathbb{R}^n$, we find from the energy conservation
law that
$$
\|\nabla\Phi\|_2^2\leq
2E(\Phi_0)+C\|\nabla\Phi\|_2^{n(p-1)/2}\|\Phi\|_2^{p+1-n(p-1)/2}.
$$
Clearly an uniform bound on $\|\nabla\Phi\|_2$ results, provided
$p<1+4/n$, and accordingly the solution exists globally. For $p\geq
1+4/n$, blow-up of the solution may occur. In fact, if there exists
a constant $\delta<0$ such that $Q(\Phi)\leq \delta <0$ or
$Q(\Phi)<0$ and $J'(0)\leq 0$ simultaneously, it's obvious from the
facts $J^{''}(t)\leq 16\delta<0$ or $J'(0)\leq 0$ and $J^{''}(t)<0$
that the solution blows up in finite time.

In the case $p\geq 1+4/n$, the sharp thresholds of blow-up and
global existence become very interesting. For the single
Schr\"{o}dinger equation, the sharp thresholds of blow-up and global
existence have been extensively studied (see the related works
\cite{H.Be,Y.L,W.S,J.Z}). Our present work in this paper is to
derive two types of sharp thresholds for the system (\ref{system}).
To our knowledge, these are the first results in this direction for
the system (\ref{system}), which seem new even for the single
nonlinear Schrodinger equation on $\mathbb{R}^n$. See Theorems
\ref{Ithm}, \ref{IIthm} below.

Recall that we have defined the $\mathcal{C}^0$ functionals $M({\bf
u})$, $E({\bf u})$ and $Q({\bf u})$ for ${\bf u}=(u_1,...,u_N)\in
(H^1(\mathbb{R}^n))^N$ in (\ref{M}), (\ref{E}) and (\ref{Q}).

\begin{Thm} (Sharp Threshold I)\label{Ithm}
Assume that $1+4/n\leq p<1+4/(n-2)^+$. The constrained variational
problem
$$
d_I:=\inf_{\{{\bf u}\in (H^1(\mathbb{R}^n))^N\setminus\{0\};\ G({\bf
u})=0\}} \frac{1}{2}\|\nabla{\bf u}\|_2^2
$$
with
\begin{align*}
G({\bf u})=(M({\bf
u}))^{p+1-n(p-1)/2}-\frac{1}{p+1}\mathcal{P}(\Phi)
\end{align*}
satisfies $d_I>0$. Besides, assume the initial data $\Phi_0\in
(H^1(\mathbb{R}^n))^N$ satisfies
$$
(M(\Phi_0))^{p+1-n(p-1)/2}+E(\Phi_0)<d_I.
$$
We have:\\
(A). If $G(\Phi_0)>0$, then the solution exists globally;\\
(B). If $G(\Phi_0)<0$, $|x|\Phi_0(x)\in (L^2(\mathbb{R}^n))^N$, and
$$ \Im\sum_{j=1}^N\int_{\mathbb{R}^n}(\nabla\phi_{0j}\cdot
x)\bar{\phi}_{0j}\leq 0$$ when $p>1+4/n$, then the solution blows up
in finite time.
\end{Thm}

\begin{Thm} (Sharp Threshold II)\label{IIthm}
Assume that $1+4/n<p<1+4/(n-2)^+$. Let $\gamma>0$ be any fixed
constant. The constrained variational problem
$$
d_{II}:=d_{II}(\gamma)=\inf_{\{{\bf u}\in
(H^1(\mathbb{R}^n))^N\setminus\{0\};\ Q({\bf u})=0\}} (M({\bf
u}))^\gamma+E({\bf u})
$$
satisfies $d_{II}>0$. Besides, assume the initial data $\Phi_0\in
(H^1(\mathbb{R}^n))^N$ satisfies
$$
(M(\Phi_0))^\gamma+E(\Phi_0)<d_{II},
$$
then we have\\
(A). If $Q(\Phi_0)>0$, the solution exists globally;\\
(B). If $Q(\Phi_0)<0$ and $|x|\Phi_0(x)\in (L^2(\mathbb{R}^n))^N$,
the solution blows up in finite time.
\end{Thm}

As corollaries, we invoke the sharp thresholds to obtain small data
criterions for the global existence of (\ref{system}). We get the
following two results.

\begin{Cor} (Small Data Criterion I)\label{Icor}
Assume that $1+4/n\leq p<1+4/(n-2)^+$. Then if the initial data
$\Phi_0\in (H^1(\mathbb{R}^n))^N$ satisfies
$$
\frac{1}{2}\|\nabla\Phi_0\|^2_2+(M(\Phi_0))^{p+1-n(p-1)/2}<d_I,
$$
the solution of (\ref{system}) exists globally.
\end{Cor}

\begin{Cor} (Small Data Criterion II)\label{IIcor}
Assume that $1+4/n<p<1+4/(n-2)^+$. Then if the initial data
$\Phi_0\in (H^1(\mathbb{R}^n))^N$ satisfies
$$
\frac{1}{2}\|\nabla\Phi_0\|^2_2+(M(\Phi_0))^\gamma<d_{II},
$$
the solution of (\ref{system}) exists globally.
\end{Cor}

\begin{Rk}
Notice that the first type of thresholds deals with $p\geq1+4/n$
while the second type only deals with $p>1+4/n$.
\end{Rk}

Both for the physical and mathematical reasons, in the last five
years, many authors paid much attention to the Cauchy problem of the
Schr\"{o}dinger equation posed on an arbitrary Riemannian manifold
$(\mathbb{M},g)$  with $\Delta_{g}$ being the associated
Laplace-Beltrami operator (where $\Delta_{g}u=u^{''}$ on the real
line $\mathbb{R}$). See the recent papers
\cite{V.B1,V.B2,N.B1,N.B2,N.B3,N.B4} and the references therein. In
the setting of $(H^1(\mathbb{M}))^N$, the conservation laws of mass
(\ref{M}) and energy (\ref{E}) hold true for (\ref{system}) on
$(\mathbb{M},g)$ with $\int_{\mathbb{R}^n}$ replaced by
$\int_{\mathbb{M}}$ (the volume integration on $\mathbb{M}$). The
virial identities (\ref{J'}) and (\ref{J"}) are also valid with
$\rho$ being a $\mathcal{C}^4$ function on $\mathbb{M}$ and
$\nabla_{\mathbb{M}}$ being the associated gradient operator
(\cite{L}).

For (\ref{system}) on $\mathbb{H}^n$ and on $\mathbb{S}^2$, a
similar local existence result as Proposition 1 still holds when we
replace $(H^1(\mathbb{R}^n))^N$ in Proposition 1 by
$(H^1(\mathbb{H}^n))^N$ and $(H^1(\mathbb{S}^2))^N$ respectively.
The readers can consult \cite{V.B2,N.B1,N.B2} for more related
discussions about the single Schr\"{o}dinger equation on
$\mathbb{H}^n$ and $\mathbb{S}^n$. The reason why we restrict
ourselves on $\mathbb{S}^2$ instead of $\mathbb{S}^n$ is that when
$n\geq 3$ the global wellposedness and the blow-up phenomena seem
more delicate than the case $n=1,2$. For $n\geq 3$, some negative
results of wellposedness on $\mathbb{S}^n$ attributed to N. Burq, P.
G$\acute{e}$rard, and N. Tzvetkov, which are in strong contrast with
the case $\mathbb{R}^n$, $\mathbb{H}^n$ and $\mathbb{S}^2$, can be
found in \cite{N.B1,N.B2,N.B4} (see also \cite{V.B1}). Our results
for the Schr\"{o}dinger system (\ref{system}) on $\mathbb{H}^n$ read
as follows. We emphasize that on $\mathbb{H}^n$ we have to make a
difference dealing with the radial case and the nonradial case, due
to the nonvanishing curvature of the manifold.

\begin{Thm}\label{H^nIthm}(Sharp Threshold I on
$\mathbb{H}^n$: Radial Case) Assume that $1+4/n\leq p<1+4/(n-2)^+$.
The constrained variational problem
$$
d_{\mathbb{H}^nI}:=\inf_{\{{\bf u}\in
(H^1(\mathbb{H}^n))^N\setminus\{0\};\ G({\bf u})=0\}}
\frac{1}{2}\|\nabla_{\mathbb{H}^n}{\bf u}\|_2^2
$$
with
\begin{align*}
G({\bf u})=(M({\bf
u}))^{p+1-n(p-1)/2}-\frac{1}{p+1}\mathcal{P}(\Phi)
\end{align*}
satisfies $d_{\mathbb{H}^nI}>0$. Besides, assume the initial data
$\Phi_0\in (H^1(\mathbb{H}^n))^N$ is radial and satisfies
$$
(M(\Phi_0))^{p+1-n(p-1)/2}+E(\Phi_0)<d_{\mathbb{H}^nI}.
$$
Then we have:\\
(A). If $G(\Phi_0)>0$, the solution exists globally;\\
(B). If $G(\Phi_0)<0$, $|x|\Phi_0(x)\in (L^2(\mathbb{H}^n))^N$, and
$$
\Im\sum_{j=1}^N\int_{\mathbb{H}^n}(\nabla_{\mathbb{H}^n}\phi_{0j}\cdot
\nabla_{\mathbb{H}^n} \rho)\bar{\phi}_{0j}\leq 0$$ when $p>1+4/n$,
the solution blows up in finite time.

Here $\rho=r^2$, where $r=r(x)$ is the geodesic distance from
$x\in\mathbb{H}^n$ to the origin $O\in\mathbb{H}^n$.
\end{Thm}

Theorem \ref{H^nIthm} doesn't work for the nonradial case. However,
the second type of thresholds on $\mathbb{H}^n$ (see Theorem
\ref{H^nIIthm}) holds for the nonradial case fortunately. To state
it, we need the following definition
\begin{align*}
Q^*(\Phi):=
\frac{1}{2}\|\nabla_{\mathbb{H}^n}\Phi\|_2^2-\frac{(n-1)(p-1)}{4(p+1)}\mathcal{P}(\Phi).
\end{align*}

\begin{Thm} (Sharp Threshold II on $\mathbb{H}^n$)\label{H^nIIthm}

{\bf Radial Case:}

Assume that $1+4/n<p<1+4/(n-2)^+$. The constrained variational
problem
$$
d_{\mathbb{H}^nII}:=\inf_{\{{\bf u}\in
(H^1(\mathbb{H}^n))^N\setminus\{0\};\ Q({\bf u})=0\}} (M({\bf
u}))^\gamma+E({\bf u})
$$
with $\gamma>0$ being an arbitrary constant satisfies
$d_{\mathbb{H}^nII}>0$. Besides, assume the initial data $\Phi_0\in
(H^1(\mathbb{H}^n))^N$ is radial and satisfies
$$
(M(\Phi_0))^\gamma+E(\Phi_0)<d_{\mathbb{H}^nII},
$$
then we have\\
(A). If $Q(\Phi_0)>0$, the solution exists globally;\\
(B). If $Q(\Phi_0)<0$ and $|x|\Phi_0(x)\in (L^2(\mathbb{H}^n))^N$,
the solution blows up in finite time.

{\bf Nonradial Case:}

Assume $n\geq 2$ and $1+4/(n-1)<p<1+4/(n-2)^+$. The constrained
variational problem
$$
d_{\mathbb{H}^nII}^*:=\inf_{\{{\bf u}\in
(H^1(\mathbb{H}^n))^N\setminus\{0\};\ Q^*({\bf u})=0\}} (M({\bf
u}))^\gamma+E({\bf u})
$$
with $\gamma>0$ being an arbitrary constant satisfies
$d_{\mathbb{H}^nII}^*>0$. Besides, assume the initial data
$\Phi_0\in (H^1(\mathbb{H}^n))^N$ satisfies
$$
(M(\Phi_0))^\gamma+E(\Phi_0)<d_{\mathbb{H}^nII}^*,
$$
then we have\\
(A). If $Q(\Phi_0)>0$, the solution exists globally;\\
(B). If $Q(\Phi_0)<0$ and $|x|\Phi_0(x)\in (L^2(\mathbb{H}^n))^N$,
the solution blows up in finite time.

\end{Thm}

Corresponding to Theorems \ref{H^nIthm}, \ref{H^nIIthm}, we have the
small data criterions below.

\begin{Cor}\label{H^nIcor}(Small Data Criterion I on $\mathbb{H}^n$)
Assume that $1+4/n\leq p<1+4/(n-2)^+$. Then if the initial data
$\Phi_0\in (H^1(\mathbb{H}^n))^N$ no matter radial or not satisfies
$$
\frac{1}{2}\|\nabla_{\mathbb{H}^n}\Phi_0\|^2_2+(M(\Phi_0))^{p+1-n(p-1)/2}<d_{\mathbb{H}^nI},
$$
the solution of (\ref{system}) exists globally.
\end{Cor}

\begin{Cor} (Small Data Criterion II on $\mathbb{H}^n$)\label{H^nIIcor}
Assume that $1+4/n<p<1+4/(n-2)^+$. Then if the initial data
$\Phi_0\in (H^1(\mathbb{H}^n))^N$ no matter radial or not satisfies
$$
\frac{1}{2}\|\nabla_{\mathbb{H}^n}\Phi_0\|^2_2+(M(\Phi_0))^\gamma<d_{\mathbb{H}^nII},
$$
the solution of (\ref{system}) exists globally.
\end{Cor}

We are now in position to state the sharp threshold for the
Schr\"{o}dinger system (\ref{system}) posed on $\mathbb{S}^2$. From
the viewpoint of geometry, the compactness of $\mathbb{S}^2$ results
in the difference between the Sobolev embedding on $\mathbb{S}^2$
and the ones on $\mathbb{R}^n$ and $\mathbb{H}^n$. To display the
same spirit as in the analysis on $\mathbb{R}^n$ and $\mathbb{S}^n$,
we prefer to work on the function space
$$
\Lambda:=\{{\bf u}\in (H^1(\mathbb{S}^2))^N\setminus\{0\};\ {\bf u}\
\textrm{is antisymmetric about the equator}\},
$$
and we define
$$
Q^{**}(\Phi):=\frac{1}{2}\|\nabla_{\mathbb{S}^2}\phi\|_2^2
-\frac{p-1}{4(p+1)}\mathcal{P}(\Phi).
$$
Our results are as below.

\begin{Thm} (Sharp Threshold on $\mathbb{S}^2$)\label{S^2thm}
Assume that $5<p<\infty$. Let $\gamma>0$ be an arbitrary constant.
The constrained variational problem
$$
d_{\mathbb{S}^2}:=\inf_{\{{\bf u}\in\Lambda;\ Q^{**}({\bf u})=0\}}
(M({\bf u}))^\gamma+E({\bf u})
$$
 satisfies $d_{\mathbb{S}^2}>0$. Besides, assume the initial data
$\Phi_0\in\Lambda$ satisfies
$$
(M(\Phi_0))^\gamma+E(\Phi_0)<d_{\mathbb{S}^2},
$$
then we have\\
(A). If $Q^{**}(\Phi_0)>0$, the solution exists globally;\\
(B). If $Q^{**}(\Phi_0)<0$, the solution blows up in finite time.
\end{Thm}

\begin{Cor} (Small Data Criterion on $\mathbb{S}^2$)\label{S^2cor}
Assume that $5<p<\infty$. Then if the initial data
$\Phi_0\in\Lambda$ satisfies
$$
\frac{1}{2}\|\nabla_{\mathbb{S}^2}\Phi_0\|^2_2+(M(\Phi_0))^\gamma<d_{\mathbb{S}^2},
$$
the solution of (\ref{system}) exists globally.
\end{Cor}

The rest of our paper is organized as follows. In section 2, we
prepare some abstract analysis for the invariant sets. In section 3,
we give the proofs of Theorems \ref{Ithm}, \ref{IIthm} and
Corollaries \ref{Icor}, \ref{IIcor}. In section 4, we give the
proofs of Theorems \ref{H^nIthm}, \ref{H^nIIthm} and Corollaries
\ref{H^nIcor}, \ref{H^nIIcor}. In section 5, we give the proofs of
Theorem \ref{S^2thm} and Corollary \ref{S^2cor}. At the end, we give
some heuristic analysis about the strong instability of the solitary
waves in Section 6.

\section{Some abstract analysis}

In this section, we will establish the invariant sets for
(\ref{system}) via the three $\mathcal{C}^0$ functionals $M({\bf
u})$, $E({\bf u})$ and $Q({\bf u})$. Our analysis can be formed as
the following proposition. Let $\mathbb{M}=\mathbb{R}^n$,
$\mathbb{H}^n$ or $\mathbb{S}^2$.

\begin{Prop}\label{inv}
(Invariant Sets) Let $F({\bf u})$ and $G({\bf u})$ be two
$\mathcal{C}^0$ functionals on $(H^1(\mathbb{M}))^N$, and $f(x,y)$
be a $\mathcal{C}^0$ function on $\mathbb{R}^2$. Suppose that the
cross-constrained minimization problem
$$
d:=\inf_{\{{\bf u}\in (H^1(\mathbb{M}))^N\setminus\{0\};\ G({\bf
u})=0\}} F({\bf u})
$$
satisfies $d>0$. If in addition
\begin{align}
G({\bf u})=0\Rightarrow F({\bf u})\leq f(M({\bf u}),E({\bf
u})),\label{F<f}
\end{align}
then the sets
$$
K_+=\{{\bf u}\in (H^1(\mathbb{M}))^N;\ G({\bf{u}})>0,\ f(M({\bf
u}),E({\bf u}))<d\}
$$
and
$$
K_-=\{{\bf u}\in (H^1(\mathbb{M}))^N;\ G({\bf{u}})<0,\ f(M({\bf
u}),E({\bf u}))<d\}
$$
are all invariant sets of the Schr\"{o}dinger system (\ref{system})
on $\mathbb{M}$.
\end{Prop}

\begin{proof}

Assume $\Phi_0\in K_+$, that is, $G(\Phi_0)>0$ and
$f(M(\Phi_0),E(\Phi_0))<d$. Noticing that $M(\Phi)$ and $E(\Phi)$
are conservation quantities for (\ref{system}), we have
$$
f(M(\Phi),E(\Phi))=f(M(\Phi_0),E(\Phi_0))<d.
$$
We now show that $G(\Phi)>0$. Otherwise, from the continuity, there
were a $t^*\in(0,T)$ such that $G(\Phi(t^*))=0$ and $\Phi(t^*)\neq
0$. We infer from (\ref{F<f}) that
$$
F(\Phi(t^*))\leq f(M(\Phi(t^*)),E(\Phi(t^*)))<d,
$$
which is a contradiction with the minimization of $d$. Thus we get
that $G(\Phi)>0$ and therefore $\Phi\in K_+$.

By the same argument, we have $K_-$ is also invariant under the flow
generated by (\ref{system}).

\end{proof}

\begin{Rk}
The idea of this proposition goes back to H. Berestycki and T.
Cazenave \cite{H.Be}. However, they restricted themselves only to
the case
$$
f(M,E)=M+E.
$$
As a consequence, they obtained the invariant sets for the
Schr\"{o}dinger equation only on $\mathbb{R}^2$. The reader will see
below that we introduce
\begin{align}
f(M,E)=M^\gamma+E\label{idea}
\end{align}
to enlarge the invariant sets of the Schr\"{o}dinger equation on
$\mathbb{R}^2$ to the Schr\"{o}dinger system (\ref{system}) on
$\mathbb{R}^n$ for all $n\geq 1$ and on some other Riemannian
manifolds. The power $\gamma>0$ in (\ref{idea}) relies heavily on
the Gagliardo-Nirenberg inequality.
\end{Rk}

Suppose we already get that
$$
K_+=\{{\bf u}\in (H^1(\mathbb{M}))^N;\ G({\bf{u}})>0,\ f(M({\bf
u}),E({\bf u}))<d\}
$$
and
$$
K_-=\{{\bf u}\in (H^1(\mathbb{M}))^N;\ G({\bf{u}})<0,\ f(M({\bf
u}),E({\bf u}))<d\}
$$
are invariant sets of (\ref{system}). Moreover, if we can show that
there exist two constants $M$, $\delta$ such that
$$
\Phi_0\in K_+\Rightarrow \|\nabla_{\mathbb{M}}\Phi\|_2\leq M<\infty
$$
and
$$
\Phi_0\in K_-\Rightarrow J^{''}(t)\leq\delta<0\ \textrm{or}\
J^{''}(t)<0\ \textrm{and}\ J'(0)\leq 0\ \textrm{simultaneously},
$$
then we arrive at the conclusion that $\Phi_0\in K_+$ implies the
solution exists globally and $\Phi_0\in K_-$ implies that the
solution blows up in finite time. In this sense, under the
assumption $f(M(\Phi_0),E(\Phi_0))<d$, we say $G(\Phi_0)=0$ is a
sharp threshold of blow-up and global existence.

\section{The proofs of Theorems \ref{Ithm}, \ref{IIthm} and Corollaries \ref{Icor}, \ref{IIcor}}

This section is devoted to the proofs of Theorems \ref{Ithm},
\ref{IIthm} and Corollaries \ref{Icor}, \ref{IIcor}. Let's  recall
the Gagliardo-Nirenberg inequality (\cite{We1}) for $1\leq
p<1+4/(n-2)^+$:
\begin{align}
\|\phi\|_{p+1}^{p+1}\leq C\|\nabla
\phi\|^{n(p-1)/2}_2\|\phi\|_2^{p+1-n(p-1)/2},\ \ \forall\ \phi\in
H^1(\mathbb{R}^n).\label{Gag}
\end{align}

{\bf The proof of Theorem \ref{Ithm}.}

Step 1. We claim that the constrained variational problem in Theorem
\ref{Ithm} satisfies $d_I>0$. For ${\bf u}\in
(H^1(\mathbb{R}^n))^N\setminus\{0\}$ subjected to $G({\bf u})=0$, it
follows from (\ref{Gag}) that
\begin{align*}
(M({\bf
u}))^{p+1-n(p-1)/2}&=\frac{1}{p+1}\mathcal{P}(\Phi)\\
&\leq C\|\nabla{\bf u}\|_2^{n(p-1)/2}(M({\bf u}))^{p+1-n(p-1)/2},
\end{align*}
which indicates $d_I>0$.

Step 2. Choosing $F({\bf u})=\frac{1}{2}\|\nabla{\bf u}\|_2^2$ and
$f(M,E)=M^{p+1-n(p-1)/2}+E$ in Proposition \ref{inv}, we see that
$$
K_+=\{{\bf u}\in (H^1(\mathbb{R}^n))^N;\ G({\bf{u}})>0,\ (M({\bf
u}))^{p+1-n(p-1)/2}+E({\bf u})<d_I\}
$$
and
$$
K_-=\{{\bf u}\in (H^1(\mathbb{R}^n))^N;\ G({\bf{u}})<0,\ (M({\bf
u}))^{p+1-n(p-1)/2}+E({\bf u})<d_I\}
$$
are invariant sets of (\ref{system}).

Step 3. Assume that $\Phi_0$ satisfies $G(\Phi_0)>0$. Then from step
2, we have $G(\Phi)>0$ and $(M(\Phi))^{p+1-n(p-1)/2}+E(\Phi)<d_I$,
which imply
$$
\frac{1}{2}\|\nabla\Phi\|_2^2<d_I,
$$
and consequently the solution exists globally.

Step 4. Assume that $\Phi_0$ satisfies $G(\Phi_0)<0$. From step 2,
we have $G(\Phi)<0$ and $(M(\Phi))^{p+1-n(p-1)/2}+E(\Phi)<d_I$.

Case(i): $p=1+4/n$. In this case, $Q(\Phi)=E(\Phi)=E(\Phi_0)$. From
$G(\Phi_0)<0$ we get that there exists a $\lambda\in(0,1)$ such that
$G(\lambda\Phi_0)=0$, that is,
\begin{align}
(M(\Phi_0))^{p-1}=\frac{\lambda^2}{p+1}\mathcal{P}(\Phi).\label{F>d>f1}
\end{align}
Then it follows from the minimization of $d_I$ that
\begin{align}
\frac{1}{2}\|\nabla(\lambda\Phi_0)\|_2^2\geq
d_I>(M(\Phi_0))^{p-1}+E(\Phi_0).\label{F>d>f2}
\end{align}

Inserting (\ref{F>d>f1}) into (\ref{F>d>f2}) yields
\begin{align*}
\frac{\lambda^2}{2}\|\nabla\Phi_0\|_2^2\geq
d_I>(M(\Phi_0))^{p-1}+E(\Phi_0)=\frac{1}{2}\|\nabla\Phi_0\|_2^2+\frac{\lambda^2-1}{p+1}\mathcal{P}(\Phi),
\end{align*}
that is, $(1-\lambda^2)E(\Phi_0)<0$. Thus we have
$J^{''}(t)=16E(\Phi_0)<0$ and therefore the solution blows up in
finite time.

Case (ii). $p>1+4/n$. In this case, for any fixed $t\in(0,T)$, there
exists a $\lambda\in(0,1)$ such that $G(\lambda\Phi)=0$, that is,
\begin{align}
(M(\Phi))^{p+1-n(p-1)/2}=\frac{\lambda^{n(p-1)/2}}{p+1}\mathcal{P}(\Phi).\label{F>d>f3}
\end{align}
Then it follows from the minimization of $d_I$ that
\begin{align}
\frac{1}{2}\|\nabla(\lambda\Phi)\|_2^2\geq
d_I>(M(\Phi))^{p+1-n(p-1)/2}+E(\Phi).\label{F>d>f4}
\end{align}

Inserting (\ref{F>d>f3}) into (\ref{F>d>f4}) yields
\begin{align*}
\frac{\lambda^2}{2}\|\nabla\Phi\|_2^2\geq
d_I&>(M(\Phi))^{p+1-n(p-1)/2}+E(\Phi)\\
&=\frac{1}{2}\|\nabla\Phi\|_2^2+\frac{\lambda^{n(p-1)/2}-1}{p+1}\mathcal{P}(\Phi),
\end{align*}
that is,
\begin{align*}
\frac{1}{2}\|\nabla\Phi\|_2^2\leq
\frac{\lambda^{-n(p-1)/2}-1}{1-\lambda^2}(M(\Phi))^{p+1-n(p-1)/2}.
\end{align*}
We infer from the above inequality that
\begin{align*}
J^{''}(t)&=Q(\Phi)\\
&=\frac{1}{2}\|\nabla\Phi\|_2^2-\frac{n(p-1)}{4}\lambda^{-n(p-1)/2}(M(\Phi))^{p+1-n(p-1)/2}\\
&\leq h(\lambda)(M(\Phi))^{p+1-n(p-1)/2}<0,
\end{align*}
with the fact
\begin{align*}
h(\lambda)=\frac{\lambda^{-n(p-1)/2}}{1-\lambda^2}(1-\lambda^{n(p-1)/2}-\frac{n}{4}(p-1)(1-\lambda^2))<0,\
\ \forall\ \lambda\in(0,1)
\end{align*}
used in the last step.

Thus we get that $J'(0)\leq 0$ and $J^{''}(t)<0$, which suggest that
the solution blows up in finite time. The proof of Theorem
\ref{Ithm} is concluded.

\hfill $\Box$

{\bf The proof of Corollary \ref{Icor}.}

From the assumption
$$
\frac{1}{2}\|\nabla\Phi_0\|^2_2+(M(\Phi_0))^{p+1-n(p-1)/2}<d_I,
$$
it's obvious that
$$
(M(\Phi_0))^{p+1-n(p-1)/2}+E(\Phi_0)<d_I.
$$
In view of Theorem \ref{Ithm}, we only have to check that
$G(\Phi_0)>0$. If else, one would have $G(\Phi_0)\leq 0$. Due to the
minimization of $d_I$, $G(\Phi_0)\neq 0$. If $G(\Phi_0)<0$, there
exists a $\lambda\in(0,1)$ such that $G(\lambda\Phi_0)=0$ and
consequently we have
$$
\frac{1}{2}\|\nabla(\lambda\Phi_0)\|_2^2\geq d_I,
$$
which is contradictory with
$$
\frac{1}{2}\|\nabla\Phi_0\|_2^2<d_I.
$$

\hfill $\Box$

{\bf The proof of Theorem \ref{IIthm}.}

Step 1. The constrained variational problem in Theorem \ref{IIthm}
satisfies $d_{II}>0$. We argue by contradiction. Suppose there
exists a sequence ${\bf u}_k\in(H^1(\mathbb{R}^n))^N\setminus\{0\}$
satisfying $Q({\bf u}_k)=0$ and $(M({\bf u}_k))^\gamma+E({\bf
u}_k)\rightarrow 0$ as $k\rightarrow 0$. By $Q({\bf u}_k)=0$ we get
that
$$
(M({\bf u}_k))^\gamma+E({\bf u}_k)=(M({\bf
u}_k))^\gamma+\frac{n(p-1)-4}{2n(p-1)}\|\nabla{\bf
u}_k\|^2_2\rightarrow 0,
$$
which indicates that $M({\bf u}_k)\rightarrow 0$ and $\|\nabla{\bf
u}_k\|_2\rightarrow 0$. On the other hand, by the
Gagliardo-Nirenberg inequality, we get from $G({\bf u}_k)=0$ that
\begin{align*}
\frac{1}{2}\|\nabla{\bf
u}_k\|_2^2=\frac{n(p-1)}{4(p+1)}\mathcal{P}(\Phi)\leq C\|\nabla{\bf
u}_k\|_2^{n(p-1)/2}(M({\bf u}_k))^{(p+1)-n(p-1)/2},
\end{align*}
that is
$$
\|\nabla{\bf u}_k\|_2^{n(p-1)/2-2}(M({\bf
u}_k))^{(p+1)-n(p-1)/2}\geq\frac{1}{2C}>0,
$$
which contradicts with $M({\bf u}_k)\rightarrow 0$ and $\|\nabla{\bf
u}_k\|_2\rightarrow 0$.

Step 2. Choosing $F({\bf u})=(M({\bf u}))^\gamma+E({\bf u})$ and
$f(M,E)=M^\gamma+E$ in Proposition \ref{inv}, we have that
$$
K_+=\{{\bf u}\in (H^1(\mathbb{R}^n))^N;\ Q({\bf{u}})>0,\ (M({\bf
u}))^\gamma+E({\bf u})<d_{II}\}
$$
and
$$
K_-=\{{\bf u}\in (H^1(\mathbb{R}^n))^N;\ Q({\bf{u}})<0,\ (M({\bf
u}))^\gamma+E({\bf u})<d_{II}\}
$$
are invariant sets of (\ref{system}).

Step 3. Assume that $\Phi_0$ satisfies $Q(\Phi_0)>0$. Then from step
2, we have $Q(\Phi)>0$ and $(M(\Phi))^\gamma+E(\Phi)<d_{II}$, which
imply
$$
\frac{n(p-1)-4}{2n(p-1)}\|\nabla\Phi\|_2^2<d_{II},
$$
and consequently the solution exists globally.

Step 4. Assume that $\Phi_0$ satisfies $Q(\Phi_0)<0$. From step 2,
we have $Q(\Phi)<0$ and $(M(\Phi))^\gamma+E(\Phi)<d_{II}$. We assert
that
$$
Q(\Phi)\leq (M(\Phi_0))^\gamma+E(\Phi_0)-d_{II}<0,
$$
following which the solution blows up in finite time.

In actuality, the fact $Q(\Phi)<0$ yields a $\lambda\in(0,1)$ such
that $Q(\lambda\Phi)=0$ and accordingly
$(M(\lambda\Phi))^\gamma+E(\lambda\Phi)\geq d_{II}$. Moreover,
$Q(\Phi)<0$ implies that $\mathcal{P}(\Phi)>0$. Next, we do
computation to achieve
\begin{align*}
&\quad(M(\Phi_0))^\gamma+E(\Phi_0)-d_{II}\\
&\geq[(M(\Phi))^\gamma+E(\Phi)]-[(M(\lambda\Phi))^\gamma+E(\lambda\Phi)]\\
&=(1-\lambda^{\gamma})(M(\Phi))^\gamma+\frac{1}{2}(1-\lambda^2)\|\nabla\Phi\|_2^2-\frac{1-\lambda^{p+1}}{p+1}\mathcal{P}(\Phi)\\
&\geq\frac{1}{2}(1-\lambda^2)\|\nabla\Phi\|_2^2-\frac{n(p-1)(1-\lambda^{p+1})}{4(p+1)}\mathcal{P}(\Phi)\\
&=Q(\Phi)-Q(\Phi_\lambda)=Q(\Phi).
\end{align*}
This concludes the proof of Theorem \ref{IIthm}.

\hfill $\Box$

{\bf The proof of Corollary \ref{IIcor}.}

From the assumption
$$
\frac{1}{2}\|\nabla\Phi_0\|^2_2+(M(\Phi_0))^\gamma<d_{II},
$$
it's obvious that
$$
(M(\Phi_0))^\gamma+E(\Phi_0)<d_{II}.
$$
In view of Theorem \ref{IIthm}, we only have to check that
$Q(\Phi_0)>0$. If else, one would have $Q(\Phi_0)\leq 0$. Due to the
minimization of $d_{II}$, $Q(\Phi_0)\neq 0$. If $Q(\Phi_0)<0$, there
exists a $\lambda\in(0,1)$ such that $Q(\lambda\Phi_0)=0$ and
consequently we have
\begin{align*}
&\quad(M(\lambda\Phi_0))^\gamma+E(\lambda\Phi_0)\geq d_{II}\\
&\Rightarrow\frac{\lambda^2}{2}\|\nabla\Phi_0\|_2^2+\lambda^\gamma
(M(\Phi_0))^\gamma\geq d_{II},
\end{align*}
which is contradictory with
$$
\frac{1}{2}\|\nabla\Phi_0\|^2_2+(M(\Phi_0))^\gamma<d_{II}.
$$

\hfill $\Box$

\section{The proofs of Theorems \ref{H^nIthm}, \ref{H^nIIthm} and Corollaries \ref{H^nIcor}, \ref{H^nIIcor}}

In this section, we focus on the Schr\"{o}dinger system on
$\mathbb{H}^n$. The Sobolev inequality on the hyperbolic space (see
\cite{E.H}) writes as
$$
\|\phi\|_{2n/(n-2)}\leq
K_n\|\nabla_{\mathbb{H}^n}\phi\|_2-\omega_n^{-2/n}\|\phi\|_2,\ \
\forall\ \phi\in H^1(\mathbb{H}^n),
$$
where $K_n$ is the best constant for the Sobolev embedding on
$\mathbb{R}^n$, and $\omega_n$ is the volume of the sphere
$\mathbb{S}^n$. By interpolation between the $L^2$ and the
$L^{2n/(n-2)}$ norms, we get the Gagliardo-Nirenberg inequality for
functions on $H^1(\mathbb{H}^n)$ for $1\leq p<1+4/(n-2)^+$:
\begin{align*}
\|\phi\|_{p+1}^{p+1}\leq C\|\nabla_{\mathbb{H}^n}
\phi\|^{n(p-1)/2}_2\|\phi\|_2^{p+1-n(p-1)/2},\ \ \forall\ \phi\in
H^1(\mathbb{H}^n).
\end{align*}

Let's firstly consider the radial case.

{\bf The proofs of Theorem \ref{H^nIthm} and the radial case of
Theorem \ref{H^nIIthm}.}

If the initial data $\Phi_0$ is radial about the origin
$O\in\mathbb{H}^n$, by the symmetry of the system (\ref{system}) we
see easily the solution $\Phi$ is also radial. We take $\rho=r^2$,
where $r=r(x)$ is the geodesic distance from $x\in\mathbb{H}^n$ to
$O\in\mathbb{H}^n$. By the noteworthy estimates (see \cite{V.B2} for
details)
\begin{align*}
\left\{\begin{array}{ll}
D^2\rho(\nabla_{\mathbb{H}^n}\phi_j,\nabla_{\mathbb{H}^n}\bar{\phi_j})\leq
2|\nabla_{\mathbb{H}^n}\phi_j|^2,\\
\Delta^2_{\mathbb{H}^n}\rho>0,\\
\Delta_{\mathbb{H}^n}\rho\geq 2n,
\end{array}
\right.
\end{align*}
we indicate from (\ref{J"}) that
\begin{align*}
J^{''}(t)\leq 16 Q(\Phi)
\end{align*}
with $Q(\Phi)$ defined as in (\ref{Q}). Then the proofs of Theorem
\ref{H^nIthm}, the radial case of Theorem \ref{H^nIIthm} proceed
exactly the same as the ones of Theorems \ref{Ithm}, \ref{IIthm}.

\hfill $\Box$

Now we turn to the nonradial case.

{\bf The proof of the nonradial case of Theorem \ref{H^nIIthm}.}

When $\Phi$ is nonradial, the crucial estimate
$$
D^2\rho(\nabla_{\mathbb{H}^n}\phi_j,\nabla_{\mathbb{H}^n}\bar{\phi_j})\leq
2|\nabla_{\mathbb{H}^n}\phi_j|^2
$$
doesn't hold. We choose another positive radial function
$$
\rho(r)=\int_0^r(\int_0^s\sinh^{n-1}\tau
d\tau)(\sinh^{n-1}s)^{-1}ds,
$$
which satisfies (see \cite{L} for details)
\begin{align*}
\left\{\begin{array}{ll}
D^2\rho(\nabla_{\mathbb{H}^n}\phi_j,\nabla_{\mathbb{H}^n}\bar{\phi_j})\leq
\frac{1}{n-1}|\nabla_{\mathbb{H}^n}\phi_j|^2,\\
\Delta_{\mathbb{H}^n}\rho=1.
\end{array}
\right.
\end{align*}
Then from (\ref{J"}) we obtain that
\begin{align*}
J^{''}(t)\leq \frac{8}{n-1} Q^*(\Phi)
\end{align*}
Following the proof of Theorem \ref{IIthm} with the modification
that $Q(\Phi)$ is substituted by $Q^*(\Phi)$ and $p>1+4/n$ is
substituted by $p>1+4/(n-1)$, we easily arrive at the conclusions of
Theorem \ref{H^nIIthm}.

\hfill $\Box$

{\bf The proof of Corollaries \ref{H^nIcor}, \ref{H^nIIcor}.}

The idea to prove Corollary \ref{H^nIcor} is the same as the proof
of Corollary \ref{Icor}. In fact, in view of Proposition \ref{inv}
in section 2, we see that
$$
K_+=\{\Phi\in (H^1(\mathbb{H}^n))^N;\ G(\Phi)>0,\
(M(\Phi))^{p+1-n(p-1)/2}+E(\Phi)<d_{\mathbb{H}^nI}\}
$$
is an invariant set under the flow generated by the Schr\"{o}dinger
system (\ref{system}) on $\mathbb{H}^n$. Once $\Phi\in K_+$, it
follows that
$$
\frac{1}{2}\|\nabla_{\mathbb{H}^n}\Phi\|<d_{\mathbb{H}^n},
$$
which yields the global existence of the solution $\Phi$. We check
as exactly as we did in the proof of Corollary \ref{IIcor} that
$$
\Phi_0\in K_+
$$
and subsequently the proof of Corollary \ref{H^nIcor} is concluded.
The proof of Corollary \ref{H^nIIcor} proceeds along the way of the
proof of Corollary \ref{IIcor} similarly, and the details are
omitted.

\hfill $\Box$

\section{The proof of Theorem \ref{S^2thm} and Corollary \ref{S^2cor}}

In this section, we complete the proofs of Theorem \ref{S^2thm} and
Corollary \ref{S^2cor}. The Sobolev embedding has its analogue on
$\mathbb{S}^n$. See the following proposition, which is taken from
\cite{E.H}.

\begin{Prop}\label{Sobprop1}
Assume $2\leq p\leq 2n/(n-2)$ when $n\geq 3$ and $2\leq p<\infty$
when $n=2$. Then for any $\phi\in H^1(\mathbb{S}^n)$, there holds
$$
(\int_{\mathbb{S}^n}|\phi|^p)^{2/p}\leq
\frac{p-2}{n\omega_n^{1-2/p}}\int_{\mathbb{S}^n}|\nabla_{\mathbb{S}^n}\phi|^2
+\frac{1}{\omega_n^{1-2/p}}\int_{\mathbb{S}^n}|\phi|^2.
$$
\end{Prop}

Hereafter we concentrate on $\mathbb{S}^2$ and work on the space
$$
\Lambda:=\{{\bf u}\in (H^1(\mathbb{S}^2))^N\setminus\{0\};\ {\bf u}\
\textrm{is antisymmetric about the equator}\}.
$$
We have the following estimate.

\begin{Prop}\label{Sobprop2}
For any function $\phi\in H^1(\mathbb{S}^n)$ which is antisymmetric
about the equator, there holds
$$
\|\phi\|_2\leq 4\|\nabla_{\mathbb{S}^2}\phi\|_2.
$$
\end{Prop}

Before proving this proposition, we list some facts which will be
used in the sequel. In the paper of the same authors \cite{L}, we
introduce the positive function
\begin{align*}
\rho(r)=\left\{\begin{array}{ll}
-2\log\cos(r/2),\ \ 0<r\leq\pi/2,\\
0,\ \ r=0.
\end{array}
\right.
\end{align*}
We cut off the sphere $\mathbb{S}^2$ along the equator into two
hemispheres $\mathbb{S}_+^2$ and $\mathbb{S}_-^2$, which contains
the north pole $N\in\mathbb{S}^2$ and the south pole
$S\in\mathbb{S}^2$ respectively. If we regard $r=r(x)$ as the sphere
distance between the point $x\in\mathbb{S}^2$ to $N\in\mathbb{S}^2$
or to $S\in \mathbb{S}^2$, then $\rho=\rho(r)$ is a $\mathcal{C}^4$
function defined on $\mathbb{S}^2_+$ radial about $N$ or defined on
$\mathbb{S^2_-}$ radial about $S$. We denote it by $\rho_+$ and
$\rho_-$ respectively. An remarkable property of $\rho_+$ and
$\rho_-$ is that
$$
\Delta_{\mathbb{S}^2}\rho_{\pm}=1.
$$
Furthermore, we have (see \cite{L} for details)
\begin{align*}
\left\{\begin{array}{ll} |\nabla_{\mathbb{S}^2}\rho_{\pm}|\leq 1,\\
D^2\rho_{\pm}(\nabla_{\mathbb{S}^2}\phi,\nabla_{\mathbb{S}^2}\bar{\phi})\leq|\nabla_{\mathbb{S}^2}\phi|^2,\
\ \forall\ \phi\in H^1(\mathbb{S}^2_{\pm}).
\end{array}
\right.
\end{align*}

We now prove Proposition \ref{Sobprop2}.
\begin{proof}

Noticing that $\phi=0$ on
$\partial\mathbb{S}^2_+=\partial\mathbb{S}^2_-$, we can use the
technique of integration by parts to obtain that
\begin{align*}
\int_{\mathbb{S}^2}|\phi|^2&=\int_{\mathbb{S}^2_+}|\phi|^2\Delta_{\mathbb{S}^2}\rho_+
+\int_{\mathbb{S}^2_-}|\phi|^2\Delta_{\mathbb{S}^2}\rho_-\\
&=-\int_{\mathbb{S}^2_+}\nabla_{\mathbb{S}^2}|\phi|^2\cdot\nabla_{\mathbb{S}^2}\rho_+
-\int_{\mathbb{S}^2_-}\nabla_{\mathbb{S}^2}|\phi|^2\cdot\nabla_{\mathbb{S}^2}\rho_-\\
&\leq
2\int_{\mathbb{S}^2_+}|\phi||\nabla_{\mathbb{S}^2}\phi|+2\int_{\mathbb{S}^2_-}|\phi||\nabla_{\mathbb{S}^2}\phi|\\
&\leq
2(\int_{\mathbb{S}^2}|\phi|^2)^{1/2}(\int_{\mathbb{S}^2}|\nabla_{\mathbb{S}^2}\phi|^2)^{1/2},
\end{align*}
which gives the desired conclusion.

\end{proof}

Combining Propositions \ref{Sobprop1} and \ref{Sobprop2}, we achieve
for any $1\leq p<\infty$, there exists a universal constant $C$ such
that
\begin{align}
\|\Phi\|_{p+1}^{p+1}\leq C\|\nabla_{\mathbb{S}^2}\Phi\|_2^{p+1},\ \
\forall\ \Phi\in\Lambda,\label{Sob}
\end{align}
which is a Sobolev type estimate. By virtue of (\ref{Sob}), we argue
as before to see that the constrained variational problem in Theorem
\ref{S^2thm} satisfies $d_{\mathbb{S}^2}>0$.

We define
$$
J(t)=\int_{\mathbb{S}^2_+}\rho_+|\Phi|^2+\int_{\mathbb{S}^2_-}\rho_-|\Phi|^2.
$$
As in \cite{L}, we get that
\begin{align}
J^{''}(t)&\leq
4(\int_{\mathbb{S}^2_+}D^2\rho_+(\nabla_{\mathbb{S}^2}\Phi,\nabla_{\mathbb{S}^2}\bar{\Phi})+
\int_{\mathbb{S}^2_-}D^2\rho_-(\nabla_{\mathbb{S}^2}\Phi,\nabla_{\mathbb{S}^2}\bar{\Phi}))\label{Q**}\\
&\quad-2\frac{p-1}{p+1}\mathcal{P}(\Phi)\nonumber\\
&\leq4\int_{\mathbb{S}^2}|\nabla_{\mathbb{S}^2}\Phi|^2-2\frac{p-1}{p+1}\mathcal{P}(\Phi)\nonumber\\
&=8Q^{**}(\Phi).\nonumber
\end{align}
In view of (\ref{Q**}), following the proof of Theorem \ref{IIthm}
and Corollary \ref{IIcor} with $p>5$ and $Q$ replaced by $Q^{**}$,
we arrive at the conclusions of Theorem \ref{S^2thm} and Corollary
\ref{S^2cor}.

\begin{Rk}
The sharp threshold of blow-up and global existence for the
Schr\"{o}dinger system (\ref{system}) posed on $\mathbb{S}^2$ with
the initial data $\Phi_0\in (H^1(\mathbb{S}^2))^N\setminus\Lambda$
leaves open.
\end{Rk}

\section {Remarks on instability of
the solitary waves }

In this section, we are concerned with the strong instability of the
solitary waves. We only consider the Schr\"{o}dinger system
(\ref{system}) on $\mathbb{R}^n$. For any $\lambda_j>0$,
$j=1,...,N$, We define
$$
M_\lambda(\Phi)=(\sum_{j=1}^N\frac{\lambda_j}{2}\int_{\mathbb{R}^n}|\phi_j|^2)^{1/2}.
$$
Noticing that as a $L^2$ norm, $M_{\lambda}(\cdot)$ is equivalent to
$M(\cdot)$, the conclusions of Theorem \ref{IIthm} still work with
$M$ replaced by $M_\lambda$. Let $\gamma=2$ in Theorem \ref{IIthm},
and we are led to the variational minimizing problem
\begin{align}
d_{II}:=\inf_{\{{\bf u}\in (H^1(\mathbb{R}^n))^N\setminus\{0\};\
Q({\bf u})=0\}} (M_\lambda({\bf u}))^2+E({\bf u}).\label{min}
\end{align}
We have proved that $d_{II}>0$. In addition, we believe that under
some reasonable assumptions, this minimization can be attained by
some function ${\bf w}\in (H^1(\mathbb{R}^n))^N\setminus\{0\}$
subjected to an Euler-Lagrangian equation. Recently, there has been
some literature on this topic, see \cite{TC.L1,B.S}. For our
purpose, we make the following assumption.

{\bf Assumption:} the minimization of (\ref{min}) is attained by
some function ${\bf w}\in (H^1(\mathbb{R}^n))^N\setminus\{0\}$,
which satisfies
\begin{align}
\Delta w_j-\lambda_jw_j+\mu_j|w_j|^{p-1}w_j+\sum_{i\neq
j}\beta_{ij}|w_i|^{(p+1)/2}|w_j|^{(p-3)/2}w_j\label{w}
\end{align}
for $j=1,...,N$.

It's obvious that $\phi_j(x,t):=e^{i\lambda_jt}w_j(x)$ is a solution
to (\ref{system}), which is called a ground solitary wave
physically. Multiplying (\ref{w}) by $\bar{w_j}$ and integrating
over $\mathbb{R}^n$ by parts, we get that
\begin{align}
S({\bf w}):=\|\nabla{\bf w}\|_2^2+2(M_\lambda({\bf
w}))^2-\mathcal{P}({\bf w})=0.\label{S}
\end{align}
Multiplying (\ref{w}) by $x\cdot\nabla\bar{w_j}$ and integrating
over $\mathbb{R}^n$ by parts, we get the Pohozaev identity
\begin{align}
(\frac{n}{2}-1)\|\nabla {\bf w}\|_2^2+n(M_\lambda({\bf
w}))^2-\frac{n}{p+1}\mathcal{P}({\bf w})=0.\label{poho}
\end{align}
Combining (\ref{S}) and (\ref{poho}), we obtain
$$
Q(\bf w)=0.
$$

After these preliminaries, we prove the following instability
theorem.

\begin{Thm}
Suppose that $1+4/n<p<1+4/(n-2)^+$ and the above {\bf Assumption}
holds. Then for any $\epsilon>0$, there exists a $\Phi_0\in
(H^1(\mathbb{R}^n))^N$ with $\|\Phi_0-{\bf w}\|<\epsilon$ such that
the solution $\Phi$ to the Schr\"{o}dinger system (\ref{system})
with initial data $\Phi_0$ blows up in finite time.
\end{Thm}

\begin{proof}
From $S({\bf w})=0$ and $Q({\bf w})=0$, we have
$$
S(k{\bf w})<0,\ \ Q(k{\bf w})<0,\ \ \forall\ k>1.
$$
On the other hand, noticing that
$$
\frac{d}{dk}\left((M_\lambda(k{\bf w}))^2+E(k{\bf
w})\right)=\frac{1}{k}S(k{\bf w})<0,\ \ \forall\ k>1,
$$
we obtain simultaneously for all $k>1$ that
\begin{align*}
\label{system}\left\{\begin{array}{ll} (M_\lambda(k{\bf
w}))^2+E(k{\bf w})<d_{II},\\
Q(k{\bf w})<0,
\end{array}
\right.
\end{align*}
which suggest from Theorem \ref{IIthm} that the solution to
(\ref{system}) with initial data $\Phi_0=k{\bf w}$ blows up in
finite time. Then any $\Phi_0=k{\bf w}$ with $1<k<1+\epsilon$ is the
desired one.

\end{proof}

Similar discussions can be made about the Schr\"{o}dinger system
(\ref{system}) posed on $\mathbb{H}^n$ and $\mathbb{S}^2$. However,
due to the loss of the variational characterizations about the
ground solitary solutions when (\ref{system}) is considered on
manifolds, we prefer not to go deep in this direction.

\end{document}